\documentclass[11pt,a4paper,reqno]{amsart}
\usepackage{vmargin,color}

\usepackage[english]{babel}
\usepackage[utf8]{inputenc}

\usepackage{bbm,amsmath,amsthm,epsfig,latexsym,marvosym}
\usepackage{amsfonts}
\usepackage{bigints}
\usepackage{amsmath}
\usepackage{amssymb}
\usepackage{tikz}
\usepackage[miktex]{gnuplottex}
\usepackage{ifthen}

\def\R{\mathbb R}
\def\N{\mathbb N}

\def\rS{\mathbb S}

\def\cal{\mathcal}

\def\E{{\cal E}}
\def\cE{\E}

\def\cG{{\cal G}}
\def\cS{\cal S}
\def\H{{\cal H}}

\def\a{\alpha}

\def\om{\omega}

\def\Om{\Omega}
\def\La{\Lambda}
\def\S{\Sigma}
\def\ov{\overline}

\def\loc{{\rm loc}}
\def\Div{{\rm div}\,}
\def\di{\, \mathrm{d}}

\def\ca{\mathbbmss{1}}

\def\00{{\bf 0}}

\newcommand{\setsep}{\;;\;}
\newcommand\Ccinfty{C_c^\infty}
\newcommand{\bd}{\partial}

\DeclareMathOperator*{\esssup}{ess\,\,sup}

\DeclareMathOperator{\Per}{Per}

\DeclareMathOperator{\gen}{gen}

\def\big{\bigskip}

\newtheorem{theorem}{Theorem}[section]
\newtheorem{corollary}[theorem]{Corollary}
\newtheorem{proposition}[theorem]{Proposition}
\newtheorem{lemma}[theorem]{Lemma}

\theoremstyle{definition}
\newtheorem{remark}[theorem]{Remark}

\numberwithin{equation}{section}
\numberwithin{figure}{section}

\providecommand{\comments}{1}        
\ifthenelse{ \equal{\comments}{1} }  
  {\newcommand{\commentz}[5]
    {\textcolor{#4}{#1{#3 #5} #2}}}
  {\newcommand{\commentz}[5]{}}

\providecommand{\changes}{1}            
\ifthenelse{ \equal{\changes}{1} }      
  {\newcommand{\bchange}[2] {\color{#2} #1 }}  
  {\newcommand{\bchange}[2]{}}
\ifthenelse{ \equal{\changes}{1} }      
  {\newcommand{\echange}[1]{#1\normalcolor}}
  {\newcommand{\echange}[1]{}}

\providecommand{\details}{1}            
\ifthenelse{ \equal{\details}{1} }      
  {\newcommand{\detail}[1]{\par{\footnotesize {\sf Detail}: #1\par\noindent}}}
  {\newcommand{\detail}[1]{}}

\author{M. Caroccia}
\address{Center for Nonlinear Analysis, Carnegie Mellon University, 5000 Forbes Avenue, Pittsburgh (PA), USA}
\email{caroccia.marco@gmail.com}

\author{S. Littig}
\address{Mathematical Institute, University of Cologne, Weyertal 86-90, 50931 Cologne, Germany}
\email{slittig@math.uni-koeln.de}

\title{The Cheeger-N-problem in terms of BV-functions}

\PassOptionsToPackage{colorlinks=false,pdfborder={0 0 0}}{hyperref}

\usepackage{hyperref}

\begin{document}

\begin{abstract}
We reformulate the Cheeger $N$ partition problem as a minimization among a suitable class of $BV$ functions. This allows us to obtain a new existence proof for the Cheeger-$N$-problem. Moreover, we derive some connections between the Cheeger-$2$-problem and the second eigenvalue of the $1$-Laplace operator.
\end{abstract}

\maketitle

\section{Introduction and Notation}

Let $\Omega\subseteq\R^n$ be open and bounded. The classical Cheeger problem is given by
\begin{equation}\label{eqn: cheeger pb}
H_1(\Om):=\inf \frac{\Per(E)}{|E|}
\end{equation}
where the infimum is taken among all measurable sets $E\subseteq\Omega$. Here $|E|$ denotes the Lebesgue measure of $E$ and
\[
\Per(E)=\sup \left\{\int_E\Div \varphi\di x\setsep \varphi\in \Ccinfty(\R^n,\R^n), |\varphi|\le 1 \right\}
\]
denotes the distributional Perimeter of $E$ in $\R^n$. Obviously, by the Gauss-Green theorem, for bounded sets $E$ with smooth boundary the quantity $\Per(E)$ coincides with $\H^{n-1}(\partial E)$.  It can be shown that, by an easy application of the compactness Theorem for sets of finite perimeter (\cite{maggibook}), the existence of a minimizer for the problem \eqref{eqn: cheeger pb} is guaranteed for any bounded set $\Om$. Any set $E\subset \Om$ such that $H_1=P(E)/|E|$ is called a \textit{Cheeger set} of $\Om$, and the quantity $H_1$ is known as the \textit{Cheeger constant} of $\Om$. This problem gained a lot of interest in the past and the properties of Cheeger sets have been studied in various context (\cite{Leo15},\cite{Pa11},\cite{LP14}, \cite{LeoNeuSar17},\cite{alter2009uniqueness},\cite{KL06}). Remarkably the Cheeger problem can be stated in an eigenvalue problem of the $1$-Laplace operator, namely the variational problem
\[
\Lambda_1:=\inf \left\{\int_{\R^n}\di |Du|\setsep \int_{\R^n} |u|\di x=1\right\}
\]
where the infimum is taken among all functions of bounded variation with $L^1(\Omega)$-support in $\Omega$ almost everywhere.\\

A generalization of the Cheeger problem, the so called Cheeger $N$-problem, was introduced and studied  in \cite{Caroccia2017} and can be stated as follows
	\begin{equation}\label{H_N quantity}
	H_N(\Om):= \inf\left\{\sum_{i=1}^N \frac{P(E_i)}{|E_i|} \ ; \ \ E_i \subset\Omega,\  |E_i\cap E_j|=\emptyset |\right\}.
	\end{equation}

Let us point out that recently, in \cite{bucur2017honeycomb}, as a consequence of a more general result, the asymptotic behavior of the above quantity (under an assumption on the convexity of the optimal cells) has been proven to be
	\[
	\lim_{N\rightarrow +\infty} \frac{H_N(\Om) }{ N^{3/2}}= \rho_0|\Om|^{-1/2}
	\]
where $\rho_0$ denotes a universal constant, namely the Cheeger constant of a unit area regular hexagon in the plane. In an even more recent  paper \cite{bucurFraga} Bucur and Fragalà where able to prove the above asymptotic behavior in full generality.
\\

By exploiting an idea that goes back to Caffarelli and Lin (see \cite{CaLi07}) we are here able to prove (see Theorem \ref{t:LN=HN}) that the quantity \eqref{H_N quantity} can be redefined as a minimum among a suitable set of BV functions. This approach, developed in Section~\ref{s:Ngeneral}, gives us an alternative point of view in looking at the problem and leads to an additional way to prove existence of Cheeger $N$-clusters. Furthermore the regularity properties known for Cheeger $N$-clusters yield regularity properties of the minimizers of the associated variational problem. 

In Section~\ref{s:N=2} we show that in case $N=2$ there is an alternative associated variational problem, closely related to the problem of the second eigenfunction of the $1$-Laplace operator. We use this equivalence to gain more insight in the ongoing investigations of the second eigenfunction of the $1$-Laplace operator.

Finally in Section~\ref{s:examples} we provide some applications of our results in specific geometric situations.

\subsection*{Notation and Conventions}

All sets we consider are assumed to be Borel sets and $|E|$ denotes Lebesgue measure of the Borel set $E$. As usual in geometric measure theory uniqueness statements are understood in the $L^1_{\textup{loc}}$- sense, \textit{i.e.} two Borel sets $E$ and $F$ are considered equivalent provided $|E\triangle F|=|E\setminus F \cup F\setminus E|=0$. The characteristic function of a set $E$ is denoted by $\ca_E$
\[
\ca_E(x)=\begin{cases}1 & x\in E\\ 0& x\not\in E\,.\end{cases}
\]

For $u\in L^1_{\loc}(\Omega)$  and $t \in\R$ let
\[
\{u>t\}:=\{x\in \Omega \setsep u(x)>t\}
\]
and analogously for "$<$".  For an $L^1(\Omega)$-function let $u^+:=u\ca_{\{u>0\}}$ and $u^-:=u^+ -u$ denote the positive and negative part. 

We are always considering an open, bounded set $\Om$ with Lipschitz boundary and we work on $BV(\Om)$ by implicitely identifying, whenever is needed this space with the space 
\[
	\big\{u\in BV(\R^n)\setsep u=0 \text{ on } \R^n\setminus \Omega\}\,.
	\]

It is well known that in terms of this identification for $\Omega$ with Lipschitz boundary the identity
\begin{align}
\int_{\R^n}\di |Du|=\int_{\Omega}\di |Du| + \int_{\bd \Omega}|u|\di \H^{n-1}\,
\label{e:TV_decomp}
\end{align}
holds for $u\in BV(\Omega)$. Moreover we write $u=(u^1,\ldots, n^N)$ for the components of an $L^1(\Omega,\R^N)$-function $u$.

\begin{remark}
The assumption that $\Omega$ has Lipschitz boundary is actually not necessary for our derivations. In the case that $\Omega$ has a boundary less regular than Lipschitz, the space $BV(\Omega)$ should be read as
\[
BV_0(\Omega):=\{u\in BV(\R^n)\setsep u=0 \text{ on $\R^n\setminus \Omega$ a.e.}\}\,.\]
Note that if $\bd\Omega$ is irregular, $BV(\Omega)$ might be strictly larger than $BV_0(\Omega)$. However, all the required properties in our further derivations (such as the compact embedding $BV_0(\Omega)\subseteq L^1(\Omega)$) can easily be proved for $BV_0(\Omega)$ as well. The essential difference is that in general formula \eqref{e:TV_decomp} makes no sense and we are limited to work with the total variation in $\R^n$ only.
\end{remark}

\section{Case $N=1$}

Let $\Omega\subseteq \R^n$ be nonempty, open and bounded with Lipschitz boundary. The classical Cheeger problem is given by
\[
H_1:=\inf \frac{\Per(E)}{|E|}\,
\]
where the infimum is taken over all measurable sets $E\subseteq \Omega$ with nonvanishing volume  $|E|$. It is well known that this problem is connected with the variational problem
\begin{align}
\Lambda_1:=\inf \left\{\frac{\int_{\R^n}\di |Du|}{\|u\|_1}\setsep u\in BV(\Omega)\setminus \{0\}\right\}\,.\label{e:VP_1Lap}
\end{align}
Minimizers of that problem are called first eigenfunctions of the $1$-Laplace operator and their existence is easily obtained by direct methods in calculus of variations using the compact embedding of $BV(\Omega)$ in $L^1(\Omega)$.

The precise connection between Cheeger sets and first eigenfunctions of the $1$-Laplacian is given by the following proposition.
\begin{proposition}\label{p:1stEF1Lap}
There holds $H_1=\Lambda_1$. Moreover a function $u\in BV(\Omega)\setminus \{0\}$ is a minimizer of \eqref{e:VP_1Lap} if and only if for almost all $t$ the sets
\[
E_t:=\begin{cases}
\{u>t\}&\text{ for  $t>0$}\\
\{u<t\}&\text{ for  $t<0$}
\end{cases}
\]
with nonvanishing volume are Cheeger sets of $\Omega$.

In particular the first eigenfunction of the $1$-Laplace operator is unique (up to scalar multiples) if and only if the Cheeger set of $\Omega$ is unique.
\end{proposition}
The idea of the proof is basically due to \cite{kawohl-f:03}, where the attention was restricted to positive minimizers. Since we have no reference for the proof in full generality and since the idea complements our further derivations, we decided to give it here.
\begin{proof}
The key point is the coarea-formula for the total variation functional
\[
\int_{\R^n}\di |Du| = \int_\R \Per(E_t)\di t\,.
\]
Usually the coarea formula makes use of super level sets both for $t>0$ and $t<0$. The validity of the formula given here is easily deduced from the standard proof of the coarea formula (cf.~\cite[Theorem 3.40]{ambrosio2000functions}, \cite[p. 185\,ff.]{EvGa91}). In contrast to the common version the sets $E_t$ match perfectly with Cavalieri`s principle
\[
\|u\|_1=\int_\Omega u^+ \di x + \int_\Omega u^-\di x = \int_0^\infty |\{u^+>t\}|\di t + \int_0^\infty|\{u^->t\}|\di t=\int_\R|E_t|\di t\,.
\]

If $E\subseteq \Omega$ is a set of finite perimeter we have $u:=\ca_E\in BV(\Omega)$, thus
\[
\Lambda_1\le \frac{\int_{\R^n}\di |Du|}{\|u\|_1}=\frac{\Per(E)}{|E|}
\]
and by arbitrariness of $E$ we derive $\Lambda_1\le H_1$.

\medskip

Let now $u$ be a minimizer of \eqref{e:VP_1Lap}, then $\Lambda_1=\frac{\int_\R \Per(E_t)\di t}{\int_\R |E_t|\di t}$, thus
\[
\int_{\R} \Per(E_t)-\Lambda_1 |E_t|\di t=0.
\]
By $\Lambda_1\le H_1\le \frac{\Per(E_t)}{|E_t|}$ the integrand turns out to be nonnegative and thus
\[
\Per(E_t)-\Lambda_1 |E_t|=0
\]
for almost all $t$ with $|E_t|\neq 0$. But this yields $\Lambda_1=\frac{\Per(E_t)}{|E_t|}\ge H_1$ for almost all $t$ which finally implies $H_1=\Lambda_1$.

If now on the other hand almost all level sets $E_t$ with positive volume of a BV-function $u$ are Cheeger sets, we have
\[
\Lambda_1\|u\|_1=H_1\int_\R |E_t|\di t = \int_\R\Per(E_t)\di t = \int_{\R^n}\di |Du|\,,
\]
i.\,e.\ $u$ is a first eigenfunction of the $1$-Laplace operator.
\end{proof}

\section{Case $N$ general}
\label{s:Ngeneral}

Let $n,N\in\N$ and $\Om\subseteq \R^n$ be nonempty, open and bounded. The Cheeger $N$-problem in $\Omega$ is given by
\begin{align}
H_N:=\inf\left\{\sum_{i=1}^N \frac{\Per(E_i)}{|E_i|}\setsep E_i\subseteq\Omega, |E_i|>0 \text{ and } |E_i\cap E_j|=0, i\neq j\right\}\,.\label{e:infC}
\end{align}

An admissible family $(E_1,\ldots, E_N)$ in \eqref{e:infC} is called cluster and the corresponding components $E_1,\ldots, E_N$ are called chambers of the cluster. A minimizing cluster in \eqref{e:infC} is called \textit{Cheeger-$N$-cluster.}

Our goal is to show that the Cheeger $N$-problem is related to a variational problem similar to what is done in Proposition~\ref{p:1stEF1Lap}. To do so, following an idea from \cite{CaLi07}, we define $\S$
\[
\S:=\bigcup_{i=1}^N\{t e_i\setsep t\ge 0\}\,
\]
as the skeleton of the positive coordinate axis of $\R^N$. Let $BV(\Omega,\S)$ be the set of those $BV(\Omega,\R^N)$-functions $u$ that take values $u(x)=(u^1(x),\ldots, u^N(x))$ in $\Sigma$ for almost every $x\in\Omega$. We define $\cE:BV(\Omega,\R^N)\to \R$ by
\begin{align}
\E(u)=\sum_{i=1}^N |Du^i|(\R^n)=\int_{\R^n} \di \|Du\|_{\star}\label{e:defcE}
\end{align}
where $\|A\|_{\star}=\sum_{i=1}^N\left(\sum_{j=1}^n a_{ij}^2\right)^{1/2}$ for a matrix $A=(a_{ij})$ in $\R^{N\times n}$. Note that this is not the usual definition of total variation of a $BV(\R^n,\R^N)$-function, since $\|\cdot\|_{\star}$ is not the usual Euclidean norm $|\cdot|$. We now consider the minimization problem
\begin{align}
\Lambda_N:=\inf \big\{\E(u)\setsep u\in BV(\Omega,\S), \|u_i\|_1=1, i=1,\ldots N\big\}\,.\label{e:infE}
\end{align}
It will turn out that, similar to Proposition~\ref{p:1stEF1Lap}, the super level sets $\{u^i>t\}$ of a minimizer of \eqref{e:infE} are connected to the Cheeger-$N$-problem.

\begin{theorem}[Existence of minimizers for \eqref{e:infE}]\label{thm existence}
For every $N\in\N$ there exists a function $u\in BV(\Om;\S)$ such that
	\[
	\E(u)=\La_N.
	\]
\end{theorem}

\begin{proof}
Clearly $\La_N<+\infty$ since, by considering a family of $N$ disjoint balls $B_1,\ldots,B_N\subset \Om$ and by defining
	\[
	u:=\Big(\frac{1}{|B_1|} \ca_{B_1},\ldots,\frac{1}{|B_N|} \ca_{B_N}\Big)
	\]
we have
	\[
	\La_N\leq \E(u)= \sum_{i=1}^N \frac{\Per(B_i)}{|B_i|}<+\infty.
	\]
Let now $(u_{k})_{k\in \N}$ be a minimizing sequence for \eqref{e:infE}. Then, for all $i=1,\ldots,N$ and ultimately all $k\in\N$ we have 
	\begin{align*}
	2\La_N\geq \E\big(u_k\big)\geq |Du^i_k|(\R^n).
	\end{align*}

By standard compact embedding arguments we may thus assume for $i=1,\ldots N$
\[
u^i_k\to: u^i\quad\text{in $L^1(\Omega)$ as $k\to \infty$}
\]
and in particular $\|u^i\|_1=1$. We may moreover assume that the convergence is point-wise and, since $\S$ is closed $u$ takes values in $\Sigma$, almost everywhere. By lower semicontinuity of the total variation we derive
\begin{align*}
\Lambda_N\le \cE(u)=\sum_{i=1}^N
|Du_k^i|(\R^n)&\le \sum_{i=1}^N\liminf_{k\to \infty}|Du_k^i|(\R^n)\\
&\le \limsup_{k\to\infty}\sum_{i=1}^N|Du_k^i|(\R^n)=\Lambda_N\,.
\end{align*}
In particular $u=(u^1,\ldots,u^N)$ is  a  desired minimizer.

\end{proof}
\begin{remark}
If $BV(\Om)$ does not embed compactly into $L^1(\Omega)$ our existence proof fails and in fact e.\,g.\ $\Lambda_N=0$ for $\Omega=\R^n$, but the infimum is not attained. To see that consider $B_r(x_1),\ldots B_r(x_N)\subset \R^n$, $N$ disjoint ball of radius $r$, and define
	\[
	u:=\frac{1}{\om_n r^n}(\ca_{B_r(x_1)},\ldots,\ca_{B_r(x_N)}),
	\]
we must than have
	\[
	\La_N\leq \E(u)=N \frac{n\om_n r^{n-1}}{\om_n r^n}=\frac{nN}{r} \ \ \ \text{for every $r>0$}.
	\]

The arbitrariness of $r$ implies $\La_N=0$, however the only function with total variation zero is the zero function, which does not satisfy the constraints on the components.
\end{remark}

\begin{theorem}\label{t:LN=HN}
There holds $\Lambda_N=H_N$. Moreover we have correspondance between the minimizers of problems \eqref{e:infE} and \eqref{e:infC} in the following sense:

An admissible competitor in \eqref{e:infE} is minimizing if and only if for almost all $t_i>0$ such that $|\{ u^i>t_i\}|>0$ ($i=1,\ldots,N)$ the family of sets 
		\begin{equation}\label{e: upper level sets}
	E_i:= \{u^i> t_i\}
		\end{equation}
provides a Cheeger $N$-cluster $E=(E_1,\ldots, E_N)$ of $\Omega$.

In particular if $E=(E_1,\ldots, E_N)$ is a Cheeger $N$-cluster, then
\begin{equation}\label{eqn: u minim}
v:=\Big(\frac{1}{|E_1|}\ca_{E_1},\ldots,\frac{1}{|E_N|}\ca_{E_N}\Big)
\end{equation}
defines a minimizer of \eqref{e:infE}.
\end{theorem}

The proof of Theorem \ref{t:LN=HN} comes easily as a consequence of the following lemma.
\begin{lemma}\label{lemma: key lemma}
Let $u$ be a minimizer for \eqref{e:infE} and $i\in \{1,\ldots, N\}$. Then for almost every $t>0$ with $|\{u^i>t\}|\neq 0$ it holds
	\[
	\frac{\Per\left(\{u^i>t\}\right)}{|\{u^i>t\}|}= |Du^i|(\R^n)\,.
	\]
\end{lemma}
\begin{proof}
Without loss of generality we consider $i=1$ and fix $0<t<T<\esssup u^1$ and let $\a>0$. We define a function $\psi:\R_{\ge 0}\to\R_{\ge 0}$ by
	\[
	\psi(s):=\begin{cases}
    s               & \text{ for } s< t \\
	t+\a(s-t)       & \text{ for } s\in [t,T)\\
	t+\a(T-t)+(s-T) & \text{ for } s\in [T,\infty)\,.
	\end{cases}	
	\]
	
\begin{figure}
\begin{tikzpicture}[domain=-1:6,yscale=.6,xscale=.6]
\draw[color=black,thick,->] (-1,0) -- (6,0);
\draw[color=black,thick,->] (0,-1) -- (0,6);
\draw[color=green!50!black,thick] (0,0)--(2,2)--(4,3)--(6,5) node[anchor=west]{$\psi$};
\draw[color=black,thin] (2,.2)--(2,-.2) node[anchor=north] {$t$};
\draw[color=black,thin] (4,.2)--(4,-.2) node[anchor=north] {$T$};
\end{tikzpicture}%
\ \ \ \ 
\begin{tikzpicture}[domain=-1:5,yscale=.6,xscale=.6]
\draw[color=black,thick,->] (-1,0) -- (5,0);
\draw[color=black,thick,->] (0,-1) -- (0,6);
\draw[color=green!50!black,thick,domain=0:4] plot[id=u]  function {-1.5*x*(x-4)}; \draw[color=green!50!black,thick] (3,5) node[anchor=west]{$u^1$};
\draw[color=black,thin](4,2)--(-.3,2) node[anchor=east] {$t$};
\draw[color=black,thin](4,4)--(-.3,4) node[anchor=east] {$T$};
\end{tikzpicture}%
\begin{tikzpicture}[domain=-1:5,yscale=.6,xscale=.6,smooth]
\draw[color=black,thick,->] (-1,0) -- (5,0);
\draw[color=black,thick,->] (0,-1) -- (0,6);
\draw[color=green!50!black,thick,domain=0:.36701] plot[id=u]  function {-1.5*x*(x-4)}; 
\draw[color=green!50!black,thick,domain=3.63299:4] plot[id=u]  function {-1.5*x*(x-4)}; 
\draw[color=green!50!black,thick,domain=.8453:3.15470] plot[id=u]  function {-1.5*x*(x-4)-1};
\draw[color=green!50!black,thick,domain=.36701:.8453] plot[id=u]  function {2 + .5*(-1.5*x*(x-4)-2)};
\draw[color=green!50!black,thick,domain=3.15470:3.63299] plot[id=u]  function {2 + .5*(-1.5*x*(x-4)-2)};
\draw[color=green!50!black,thick] (3,4) node[anchor=west]{$\psi\circ u^1$};%
\draw[color=black,thin](4,2)--(-.3,2) node[anchor=east] {$t$};
\draw[color=black,thin](4,3)--(-.3,3) node[anchor=east] {$t+\a(T-t)$};
\end{tikzpicture}
\caption{The function $\psi$ and how it acts composed with $u^1$.}
\end{figure}
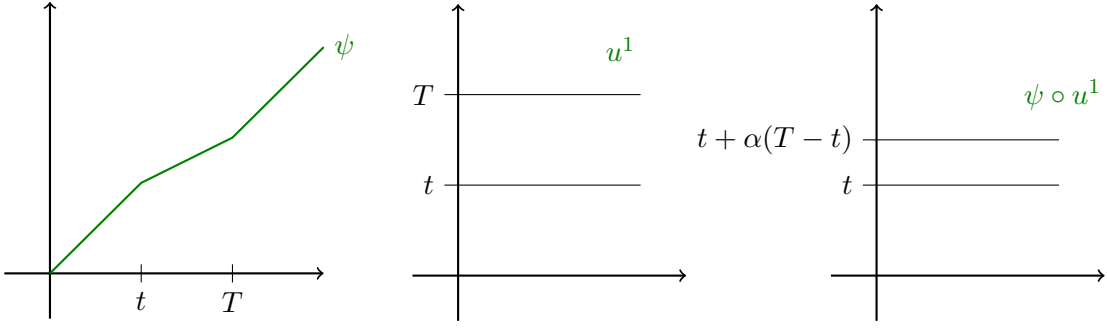	
	\[
	v:=\psi \circ u^1, \ \ \ \kappa:=\left(\int_{\Om} \psi(u^1(x))\di x\right)^{-1}, \ \ \ \overline{u}_1=\kappa v.
	\]
Notice that $\overline{u}=(\ov{u}^1,u^2,\ldots,u^N)$ is an admissible competitor in the minimization problem \eqref{e:infE} and, thanks to the minimality property of $u$, $\E(u)\leq \E(\ov{u})$, thus
	\begin{equation}\label{e: minimal proof}
 \int_{\R^n} \di |Du^1|\leq \int_{\R^n} \di |D\ov{u}^1|.
	\end{equation}

The co-area formula implies that
	\begin{align}
\int_{\R^n} \di |Dv| &=\int_0^{+\infty} \Per\left(\{v>\tau\}\right)\di \tau\nonumber\\
	&= \bigg[\int_0^{t} \Per\left(\{v>\tau\}\right)\di \tau  + \int_t^{t+\a(T-t)} \Per\left(\{v>\tau\}\right)\di \tau\nonumber\\
	& \ \ \   
	 +\int_{t+\a(T-t)}^{+\infty} \Per\left(\{v>\tau\}\right)\di \tau \bigg]\nonumber\\
	 &= \bigg[\int_0^{t} \Per\left(\{u^1>\tau\}\right)\di \tau  + \a \int_t^{T} \Per\left(\{u^1>\tau\}\right)\di \tau\nonumber\\
	& \ \ \   
	 +\int_{T}^{+\infty} \Per\left(\{u^1>\tau\}\right)\di \tau \bigg], \label{e: one}
	\end{align}
and obviously $\int_{\R^n} \di |D\ov{u}^1|=\kappa \int_{\R^n}\di |D v|$.

Moreover, thanks Cavalieri`s principle, we reach
	\begin{align}
	\kappa&=\left(\int_{\Om} v \di x\right)^{-1} =\left(\int_0^{+\infty} |\{v>\tau\}|\di \tau\right)^{-1}\nonumber\\
	&=\left(\int_0^t |\{u^1>\tau\}|\di\tau+\a \int_t^T |\{u^1>\tau\})|\di\tau +\int_T^{+\infty} |\{u^1>\tau\}|\di\tau\right)^{-1}\,.\label{e: three}
	\end{align}
By collecting \eqref{e: one}, \eqref{e: three} and by exploiting \eqref{e: minimal proof} we are lead to the relation
	\begin{align*}
\int_{\R^n} \di |D\ov{u}^1|&= \frac{\int_{(0,t)\cup(T,\infty)}\Per\left(\{u^1>\tau\}\right)\di \tau+\a\int_t^T \Per\left(\{u^1>\tau\}\right)\di\tau }{\int_{(0,t)\cup(T,\infty)} \left|\{u^1>\tau\}\right|\di\tau +\a \int_t^T \left|\{u^1>\tau\}\right|\di\tau} \\
 \geq &\frac{\int_{(0,t)\cup(T,\infty)} \Per\left(\{u^1>\tau\}\right)\di\tau+\int_t^T \Per\left(\{u^1>\tau\}\right)\di\tau }{\int_{(0,t)\cup(T,\infty)} \left|\{u^1>\tau\}\right|\di\tau+ \int_t^T \left|\{u^1>\tau\}\right|\di\tau}=\int_{\R^n}\di |Du^1|.
	\end{align*}
Notice that this is an inequality of the type
	\[
	\frac{A+\a B}{C+\a D}\geq \frac{A+B}{C+D}
	\]
where $A,B,C,D>0$ and we want it to hold for any $\a>0$. This inequality implies immediately that
	\[
	AD(1-\a)\geq BC(1-\a),
	\]
which, thanks to the arbitrariness of $\a$, leads to $AD=BC$. In particular
	\[
\frac{\int_{(0,t)\cup(T,\infty)} \Per\left(\{u^1>\tau\}\right)\di \tau}{\int_{(0,t)\cup(T,\infty)} \left|\{u^1>\tau\}\right|\di\tau} 
 =\frac{\int_t^T \Per\left(\{u^1>\tau\}\right)\di\tau }{\int_t^T \left|\{u^1>\tau\}\right|\di\tau}.
	\]	
By noticing that, for almost every $T>0$, it must hold
	\[
	\lim_{t\rightarrow T^-}\frac{\int_t^T \Per\left(\{u^1>\tau\}\right)\di\tau }{\int_t^T \left|\{u^1>\tau\}\right|\di\tau}=\frac{\Per\left(\{u^1>T\}\right)}{\left|\{u^1>T\})\right|},
	\] 
we achieve the proof.
\end{proof}

\begin{proof}[Proof of Theorem \ref{t:LN=HN}]
Let $u$ be a minimizer of problem \eqref{e:infE} and let $t_i$ be such that the statement of Lemma~\ref{lemma: key lemma} applies for each $i\in\{1,\ldots N\}$. Define $\tilde u$ by 
\[
\tilde u:= \frac{1}{|\{u^i>t_i\}|}\ca_{\{u^i>t_i\}}\,,
\] then due to the minimality of $u$ we have
\[
\Lambda_N=\cE(u)\le\E(\tilde u)= \sum_{i=1}^N \frac{\Per(\{u^i>t_i\} ) }{|\{u^i>t_i\} |}=\E(u)\,.\]
In particular we derive $H_N\le \Lambda_N=\cE(u)$.

\bigskip

Let conversely $u$ be such that for almost all $t>0$ and certain $t_2,\ldots,t_N$ the family of sets $(F_t,E_2,\ldots E_N)$ with $E_i$ as in \eqref{e: upper level sets} provides a Cheeger-$N$-cluster of $\Omega$. Since all but the first chamber are fixed and all the clusters are supposed to be optimal, we conclude that there is $h_1$ with
\[
h_1=\frac{\Per(F_t)}{|F_t|}
\]
for almost all $t$ with $|F_t|\neq 0$.
Repeating this argument for $i=2,\ldots, N$, we derive that there are $h_1,\ldots, h_N$ with
\[
h_i=\frac{\Per (\{u^i>t_i\}}{|\{u^i> t_i\}|}
\]
for almost all $t_i>0$ such that $|\{u^i> t_i\}|\neq 0$. Moreover
\[
H_N=\sum_{i=1}^N h_i\,.
\]

Using definition \eqref{e:defcE}, the co-area formula and Cavalieri's principle we now conclude
\begin{align*}
\Lambda_N\le\cE(u)=&\sum_{i=1}^N\int_0^\infty \Per(\{u^i>t\})\di t
=\sum_{i=1}^N\int_0^\infty h_i |\{u^i>t\}|\di t
=\sum_{i=1}^N h_i = H_N\,.
\end{align*}

The statement on \eqref{eqn: u minim} is elementary now.
\end{proof}

\section{Special case $N=2$}
\label{s:N=2}

The skeleton $\Sigma$ in the construction above is chosen in order to ensure the sublevel sets of the components $u^i$ to be disjoint. In case of $N=2$ we can use signed functions instead of the rather technical set $\Sigma$. Let $\cE_1$ denote the functional
\[
\cE_1(v):=\int_{\R^n}\di |Dv|
\]
for functions $v\in BV(\Omega)$. We than consider
\begin{align}
M_2:=\inf\left\{\cE_1(v)\setsep v\in BV(\Omega), \|v^+\|_1=\|v^-\|_1=1\right\}\label{e:la_2}
\end{align}
By standard compactness arguments as in \cite{kawohl-s:07} it is not difficult to see, that a minimizer of \eqref{e:la_2} exists.

Before we come to the next proposition recall the definition of $H_2$ in \eqref{e:infC} and $\Lambda_2$ in \eqref{e:infE} and their equality by Theorem~\ref{t:LN=HN}.

\begin{proposition}\label{p:4.1}
There holds
\[
\Lambda_2 = M_2\,.
\]
Moreover if $u=(u^1,u^2)$ is a Minimizer of \eqref{e:infE}, then $v=u^1-u^2$ is a minimizer of \eqref{e:la_2}. On the other hand if $v$ is a minimizer of \eqref{e:la_2}, then $u=(v^+,v^-)$ is a minimizer of \eqref{e:infE}.
\end{proposition}
\begin{proof}
If $u=(u^1,u^2)$ is an admissible function in \eqref{e:infE}, then (thanks to the fact that $u$ takes values in $\Sigma$ almost everywhere) the supports of $u^1$ and $u^2$ are essentially disjoint. Thus $v=u^1-u^2$ is a $BV(\Omega)$-function and the coarea forumla implies the relation
\begin{align}
\cE(u)=\cE_1(v)\,.\label{e:E(u)=TV(v)}
\end{align}
Moreover obviously $v^+=u^1$ and $v^-=u^2$, thus $\|v^+\|_1=\|v^-\|_1=1$, i.e. $v$ is an admissible function for \eqref{e:la_2}.

On the contrary if $v$ is an admissible function for \eqref{e:la_2}, then obviously $v^+$ and $v^-$ have essentially disjoint supports and thus $u=(v^+,v^-)$ turns out to be an admissble function for \eqref{e:infE} and again by the coarea formula \eqref{e:E(u)=TV(v)} holds. In particular we achieve $\Lambda_2=M_2$.
\end{proof}

Let us remark, that the constraints
\[
\|u^+\|_1=1\quad \text{ and }\quad \|u^-\|_1=1
\]
are equivalent to
\begin{align}
\cG_1(u):=\int_\Omega|u|\di x=2 \quad\text{and}\quad \cG_2(u):=\int_\Omega u \di x=0\,.\label{e:equivConstr}
\end{align}

\medskip

Recall that eigenfunctions (normalized with $\cG_1(u)=1$) of the $1$-Laplace operator are defined as critical points of the constraint variational problem
\[
\inf\{\cE_1(v)\setsep u\in BV(\Omega),\, \cG_1(v)=1\}
\]
and the corresponding critical values $\cE_1(v)$ are called eigenvalues of of the $1$-Laplace operator. Note that neither $\cE_1$ nor $\cG_1$ are differentiable, such that the concept of the weak slope is applied to define critical points in that context (cf. \cite{kawohl-s:07}, \cite{milbers-s:10}, \cite{chang:09}, \cite{milbers-s:12}). It has been shown in \cite{kawohl-s:07}  that eigenfunctions satisfy the following Euler-Lagrange equation. For any normalized eigenfunction $v$ there exists a vector field $z\in L^{\infty}(\Omega,\R^n)$ with
\[
|z|\le 1\quad \text{ and }\quad \Div z\in L^{n/(n-1)}(\Omega)\quad\text{ and }\quad \cE_1(v)=-\int_\Omega\Div z\, v \di x
\] and a function $s\in L^\infty(\Omega)$ with
\[
|s|\le1\quad\text{ and }\quad \cG_1(u) = \int_\Omega v s \di x 
\]
such that with $\lambda=\cE_1(v)$ the equation
\[
-\Div z = \lambda s
\]
holds. This equation gives sense to the formal Euler-Lagrange equation
\[
-\Div \frac{Dv}{|Dv|}=\lambda \frac{v}{|v|}\,.
\]
However, the Euler-Lagrange equation itself turned out to be inappropriate to define eigensolutions of the $1$-Laplace operator since it has too many solutions (that are not critical points in the sense of the weak slope), cf.~\cite{milbers-s:12}.

It is still an exciting ongoing question how to properly define or characterize higher eigenfunctions of the $1$-Laplace operator.
The most common way is in terms of a variational procedure and makes use of the genus as topological index. Recall that a symmetric set $S$ is said to be of genus $k\in\N$, denoted $\gen S=k$, provided there exists an odd continuous map $\Phi: S\to \R^k\setminus\{0\}$ and $k$ is the smallest possible integer with that property.

It is well known that 
\[
\lambda_{2,var}:=\inf_{S\in \cS_2}\sup_{v\in S}\cE_1(v)\,,
\]
where
\[
\cS_2:=\{S\subset L^1(\Omega)\setsep S=-S,\ S \text{ compact in $L^1(\Omega)$}, \gen S\ge 2, \cG_1=1 \text{ on } S\}
\]
is an eigenvalue of the $1$-Laplace operator \cite[Prop.~2.1]{littig-s:13}.

We will use our previous results to provide lower bounds on the value $\lambda_{2,var}$.

\medskip

Let now $s\in L^{\infty}(\Omega)$. We consider the minimization problem
\begin{align}
\lambda_{s}:= \inf\left\{\cE_1(v)\setsep v\in BV(\Omega), \cG_1(v)=1, \int_\Omega s v \di x =0\right\}\,.\label{e:la_s}
\end{align}

\begin{proposition}
For any $s\in L^\infty(\Omega)$ holds $\lambda_{s}\le \lambda_{2,var}$\,.
\end{proposition}
\begin{proof}
Let $\varepsilon>0$ and let $S\subseteq L^1(\Omega)$ with $\gen S\ge 2$ such that $\sup_{v\in S}\cE_1(v)\le\lambda_{2,var}+\varepsilon$. We define an odd continuous functional $\Phi: S\to \R$ by $v \mapsto \int_\Omega v s \di x$. Since $\gen S\ge 2$ there has to be an element $v_0\in S$ with $\int_\Omega v_0 s \di x=0$ (otherwise we would have $\gen S=1$). Since $v_0$ is admissible in \eqref{e:la_s} we obtain
\[
\lambda_s\le \cE_1(v_0)\le \sup_{v\in S}\cE_1(v)\le\lambda_{2,var}+\varepsilon\,.
\]
\end{proof}

\medskip
\begin{corollary}\label{c:M2/2<la2}
There holds
\[
\frac{M_2}{2}\le \lambda_{2,var}\,.
\]
\end{corollary}
\begin{proof}
This is a direct consequence of \eqref{e:equivConstr}, the $1$-homogeneity of $\cG_1$ and $\cE_1$ and the previous proposition applied with $s=1$.
\end{proof}

\begin{remark}\label{r:4.4}
Let us remark that E.~Parini in \cite{Pa09} studied the problem
\[
h_2:=\inf \left\{\max\left\{\frac{\Per (E_1)}{|E_1|},\frac{\Per (E_2)}{|E_2|}\right\}\setsep E_1, E_2\subseteq\Omega, |E_1|\,|E_2|>0 \text{ and } |E_1\cap E_2|=0\right\}\,.
\]
He proved, that the second variational eigenvalue of the $p$-Laplace operator converges as $p\to 1$ to $h_2$. Obviously $h_2\ge \frac{H_2}{2}$ and thus the statement of Corollary~\ref{c:M2/2<la2} can be recovered also from the convergence of the eigenvalues of the $p$-Laplacian to the eigenvalues of the $1$-Laplacian as $p\to 1$(cf.~\cite{littig-s:13}).
\end{remark}

\medskip

Let in the following $\rS^1$ denote the boundary of the unit disc in $\R^2$.
\begin{corollary}\label{c:sup_le_La_2}
If there exists an odd continuous map $\Phi: \rS^1\to L^1(\Omega)$  with $\|\Phi(x)\|_1=1$ for all $x\in\rS^1$ and with
\[
\sup_{x\in \rS^1} \cE_1(\Phi(x))\le \frac{M_2}{2}\,,
\]
then
\[
\frac{M_2}{2}= \lambda_{2,var}\,.
\]
\end{corollary}
\begin{proof}
It is an easy exercise to show that $S:= \Phi(\rS^1)$ has category two. Thus
\[
\lambda_{2,var}\le\sup_{v\in S}\cE_1(v)\le \frac{M_2}{2}\,.
\]
The reverse inequality follows from the previous corollary.
\end{proof}
After an explicit calculation of the geometric constant $\frac{\Lambda_2}{2}$ such a closed curve as in the previous corollary might be calculated explicitly, depending on the geometry of $\Omega$. This yields a way to calculate the second eigenvalue of the $1$-Laplace operator.

\section{Examples}
\label{s:examples}

Consider the rather easy example of a barbell domain in $\R^2$ with non-unique Cheeger sets, say e.\,g.\
\[
\Omega= [0,1]\times [0,1] \ \cup\ [1,2]\times [0,\varepsilon]\  \cup\  [2,3]\times [0,1]
\]
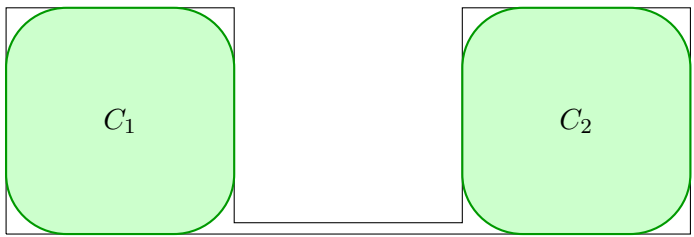
\begin{figure}
\begin{tikzpicture}[scale=3]
\draw(0,0)--(3,0)--(3,1)--(2,1)--(2,.05)--(1,.05)--(1,1)--(0,1)--(0,0);

\fill[color=green!20!white](0,.265) arc(180:270:.265)--(0.735,0) arc(270:360:.265)-- (1,.735) arc(0:90:.265)--(.265,1) arc(90:180:.265)-- cycle;

\draw[color=green!60!black,thick] (0,.265) arc(180:270:.265)--(0.735,0) arc(270:360:.265)-- (1,.735) arc(0:90:.265)--(.265,1) arc(90:180:.265)-- cycle;

\draw (.5,.5) node{$C_1$};


\fill[color=green!20!white] (2,.265) arc(180:270:.265)--(2.735,0) arc(270:360:.265)-- (3,.735) arc(0:90:.265)--(2.265,1) arc(90:180:.265)-- cycle;

\draw[color=green!60!black,thick] (2,.265) arc(180:270:.265)--(2.735,0) arc(270:360:.265)-- (3,.735) arc(0:90:.265)--(2.265,1) arc(90:180:.265)-- cycle;

\draw (2.5,.5) node{$C_2$};

\end{tikzpicture}%
\caption{Non-unique Cheeger sets of a barbell domain}
\end{figure}
with $\varepsilon>0$ sufficiently small. It is well known, that $C_1$, the Cheeger set of $[0,1]^2$ as well as $C_2:=(2,0)+C_1$, the Cheeger set of $[2,3]\times [0,1]$, and their union $C_1\cup C_2$ are Cheeger sets of $\Omega$. According to Proposition~\ref{p:1stEF1Lap} any function
\[
u_{\alpha,\beta}=\alpha \ca_{C_1} + \beta \ca_{C_2}\,,
\]
with $\alpha\neq 0$ or $\beta\neq 0$ is a minimizer of \eqref{e:VP_1Lap}. In particular in contrast to the first eigenfunction of the $p$-Laplace operator the first eigenfunction of the $1$-Laplace operator may change sign in $\Omega$ and is not unique in general.

Moreover we have $\Lambda_1=\frac{\Lambda_2}{2}=\lambda_{2,var}$ in that case. This follows from the fact that $(C_1,C_2)$ is obviously a Cheeger-$2$-cluster for $\Omega$ and thus $\Lambda_1=\frac{\Lambda_2}{2}$. It remains to show $\lambda_{2,var}\le \frac{\Lambda_2}{2}$. This will follow from Proposition~\ref{p:4.1} and Corollary~\ref{c:sup_le_La_2}. In fact the set
\[
S:=\{u_{\alpha,\beta}\setsep \alpha,\beta\in\R,\ |\alpha||C_1|+|\beta||C_2|=1\}
\]
is easily seen to be homeomorphic to $\rS^1$ and $\|u_{\alpha,\beta}\|_1=1$ for all $u_{\alpha,\beta}\in S$. It remains to estimate the energy
\begin{align*}
\cE_1(u_{\alpha,\beta})&\le \cE_1(\alpha\ca_{C_1})+\cE_1(\beta\ca_{C_2})\\
&=|\alpha|\cE_1(\ca_{C_1})+|\beta|\cE_1(\ca_{C_2})\\
&=|\alpha|\Per(C_1) + |\beta| \Per(C_2)\\
&=|\alpha|\frac{\Lambda_2}{2 |C_1|} + |\beta|\frac{\Lambda_2}{2 |C_2|}=\frac{\Lambda_2}{2}\,.
\end{align*}

Note that the foregoing calculations apply in any situation where we have a Cheeger-$2$-cluster $(C_1,C_2)$ with $\frac{\Per(C_1)}{|C_1|}=\frac{\Per(C_2)}{|C_2|}$. Recalling the results from Section~\ref{s:N=2} this directly gives the following Theorem.

\begin{theorem}
If there is a Cheeger-$2$-cluster $(C_1,C_2)$ of $\Omega$ satisfying
\[
\frac{\Per(C_1)}{|C_1|}=\frac{\Per(C_2)}{|C_2|}\,,
\]
then the second variational eigenvalue $\lambda_{2,var}$ of the $1$-Laplace operator is characterized by the following equalities
\[
\lambda_{2,var}=\frac{M_2}{2}=\frac{\Lambda_2}{2}=\frac{H_2}{2}=\frac{\Per(C_1)}{|C_1|}\,.
\]
\end{theorem}

According to Remark~\ref{r:4.4} we have $\lambda_{2,var}=h_2$ and obviously $h_2=\frac{H_2}{2}$ if and only if $\frac{\Per(C_1)}{|C_1|}=\frac{\Per(C_2)}{|C_2|}$ for at least one Cheeger-$2$-cluster of $\Omega$. There are not many specific Cheeger-$2$-clusters studied, yet. However, it is quite reasonable that the Cheeger-$2$-cluster $(C_1,C_2)$ of a slightly non-symmetric barbell domain
\[
\Omega= [0,1]\times [0,1] \,\cup\,[1,2]\times [0,\varepsilon]\, \cup\, [2,3-\delta]\times [0,1-\delta]
\]
\begin{figure}
\begin{tikzpicture}[scale=3]
\draw(0,0)--(2.8,0)--(2.8,.8)--(2,.8)--(2,.05)--(1,.05)--(1,1)--(0,1)--(0,0);

\fill[color=green!20!white](0,.265) arc(180:270:.265)--(0.735,0) arc(270:360:.265)-- (1,.735) arc(0:90:.265)--(.265,1) arc(90:180:.265)-- cycle;

\draw[color=green!60!black,thick] (0,.265) arc(180:270:.265)--(0.735,0) arc(270:360:.265)-- (1,.735) arc(0:90:.265)--(.265,1) arc(90:180:.265)-- cycle;

\draw (.5,.5) node{$C_1$};


\fill[color=green!20!white] (2,.212) arc(180:270:.212)--(2.588,0) arc(270:360:.212)-- (2.8,.588) arc(0:90:.212)--(2.212,.8) arc(90:180:.212)-- cycle;

\draw[color=green!60!black,thick] (2,.212) arc(180:270:.212)--(2.588,0) arc(270:360:.212)-- (2.8,.588) arc(0:90:.212)--(2.212,.8) arc(90:180:.212)-- cycle;

\draw (2.4,.4) node{$C_2$};

\end{tikzpicture}%
\caption{Non-symmetric barbell domain with $\lambda_{2,var}>\frac{\Lambda_2}{2}$.}
\end{figure}
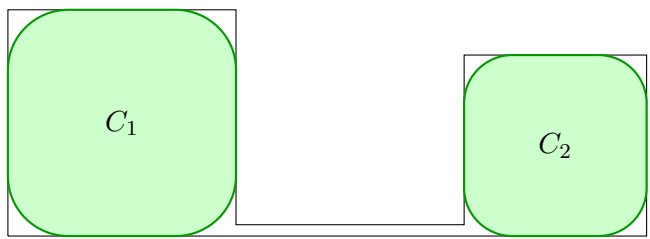
for a certain sufficiently small $\delta>0$ and an $0<\varepsilon\ll\delta$ will be given by 
$C_1$ the Cheeger set of $[0,1]\times [0,1]$ and $C_2$ the Cheeger set of $[2,3-\delta]\times [0,1-\delta]$. By the scaling properties of the Cheeger quotient we have
\[
\frac{\Per(C_1) }{|C_1|}<\frac{\Per(C_2) }{|C_2|}
\]
so
\[
\lambda_{2,var}=h_2>\frac{H_2}{2}=\frac{M_2}{2}=\frac{\Lambda_2}{2}
\]
In particular
\[
\inf\left\{\cE_1(v)\setsep \int_\Omega |v|\di x=1,\ \int_\Omega v\di x=0\right\}
\]
provides only a strict lower bound on the second variational eigenvalue $\lambda_{2,var}$ of the $1$-Laplace operator. 

\subsection*{Acknowledgement}
The work of Marco Caroccia  was supported by the Fundação para a Ciência e a Tecnologia (Portuguese Foundation for Science and Technology) through the Carnegie Mellon\textbackslash Portugal Program under Grant 18316.1.5004440. Marco Caroccia would also like to acknowledge for the hospitality the Universität zu Köln where part of this work has been developed.


\addcontentsline{toc}{chapter}{Bibliography} 

\end{document}